\documentclass[11pt]{amsart}
\usepackage{amssymb}
\usepackage{amsfonts}
\usepackage{amsaddr}

\usepackage{enumerate}

\usepackage{setspace}
\newcommand{\Z}{\mathbb{Z}}
\newcommand{\N}{\mathbb{N}}
\newcommand{\C}{\mathbb{C}}
\newcommand{\Dr}{\mathcal{D}_r}
\newcommand{\Drs}{\mathcal{D}_{r^s}}

\newcommand{\Brs}{\mathcal{B}_{r,s}}
\newcommand{\Br}{\mathcal{B}_{r}}

\newtheorem{thm}{Theorem}[section]
\newtheorem{cor}[thm]{Corollary}

\newtheorem{prop}[thm]{Proposition}

\newtheorem{lemma}[thm]{Lemma}
\newtheorem{res}[thm]{Result}
\newtheorem{defn}[thm]{Definition}

\title[DFT of $(r,s)-$even functions]{The discrete Fourier transform of $\mathbf{(r,s)-}$even functions}

\author{K Vishnu Namboothiri}
\address{Department of Collegiate Education, Government of Kerala, India}
\address{Department of Mathematics, Government Polytechnic College, Vennikulam, Pathanamthitta, Kerala - 689 544, India}
\email{kvnamboothiri@gmail.com}

\begin{document}

\begin{abstract}
An $(r,s)$-even function is a special type of periodic function mod $r^s$. These functions were defined and studied for the the first time by McCarthy. An important example for such a function is a generalization of Ramanujan sum defined by Cohen. In this paper, we give a detailed analysis of DFT of $(r,s)$-even functions and use it to prove some interesting results including a generalization of the H\"{o}lder identity. We also use DFT to give shorter proofs of certain well known results and identities .  

\end{abstract}

\keywords{ Periodic functions, $(r,s)-$even functions, multiplicative functions, discrete Fourier transform, generalized Ramanujan sum, linear congruences, H\"{o}lder's identity}
\subjclass[2010]{11A25, 11L03} 

\maketitle

\section{Introduction}
This article is a study on the discrete Fourier transform (DFT) and its various applications on a special class of periodic fuctions called the $(r,s)$-even functions. These functions were  defined by McCarthy in \cite{mccarthy1960generation}. He called the collection of such functions to be \emph{$E_s$ class}. It is a generalization of the $r$-even functions which was introduced and studied by Cohen \cite{cohen1955class, cohen1958representations}.  $(r,s)$-even functions are $r^s$-periodic and $(r,1)$-even functions are $r$-even functions itself.

The applications of DFT of periodic functions has found an important place in various fields like number theory, digital communications engineering, vibration analysis, digital signal and image processing etc. Some interesting results using DFT in signal processing is given by Samadi \emph{et al.} \cite{samadi2005ramanujan}. DFT was applied to study the GCD function (which is $r$-even) by Schramm \cite{schramm2008fourier}. Certain number theoretic aspects of DFT can be found in Haukkanen \cite{haukkanen2012discrete} and  Beck and Halloran \cite{beck2010finite}. A detailed study of DFT right from the ground level leading up to various applications can be found in the books of Sundararajan \cite{sundararajan2001discrete} and Terras \cite{terras1999fourier}.

There have been many parallels between the studies of $r$-even functions and $(r,s)$-even functions. For example, Cohen  \cite{cohen1955class} observed that the classical Ramanujan sum $c_r(n)$ which is defined as the sum of $n^{\text{th}}$ powers of $r^{\text{th}}$ primitive roots of unity is $r$-even. Parallel to this, it follows from \cite[Lemma 2]{cohen1950extension} that the generalization of the Ramanujan sum $c_{r,s}(n)$ given in \cite{cohen1949extension} is $(r,s)$-even. Cohen gave a discrete Fourier expansion of $r$-even functions in \cite[Theorem 1]{cohen1955class}. The $(r,s)$-even functions were given an analogous expansion by McCarthy in \cite[Theorem 1]{mccarthy1960generation}. 

The DFT of $r$-even functions was given a detailed analysis by T\'{o}th and Haukkanen  in \cite{toth2011discrete}. There they gave a number of new results and shorter proofs of some existing results using DFT. Parallel to this, in this paper we analyse the theory of the DFT of $(r,s)$-even  functions. We give an inversion formula for the $(r,s)$-even  functions and  prove a result showing that an $(r,s)$-even function which is multiplicative in one variable $r$ is infact multiplicative in two variables. We also give a generalization of the H\"{o}lder identity and give simpler proofs for certain existing and known results. 
\section{Notations and basic results}

Let  $s\in\N$. For $a,b\in \Z$ with atleast one of them non zero, the generalized gcd of these numbers $(a,b)_s$ is defined to be the largest $l^s\in\N$ dividing $a$ and $b$ simultaneously. Therefore $(a,b)_1=(a,b)$, the usual gcd of two integers. The set of all arithmetic functions will be denoted by $\mathcal{F}$. The arithmetic function $\tau(n)$ will denote the number of positive divisors of $n$. An arithmetic function is said to be $(r,s)$-even  if $f((n,r^s)_s)=f(n)$ for all $n\in\N$. An $(r,1)$-even function turns out to be $r-$even.

The set of all $r$-periodic functions, $r$-even functions and $(r,s)$-even  functions will be denoted by $\Dr$, $\Br$, and $\Brs$ respectively. Note that $\Dr \supset \Br$, $\Drs\supset \Dr$ and $\Drs\supset \Brs$. $\Dr$ is an $r$-dimensional subspace of the $\C$-linear space $\mathcal{F}$. It has the functions $\delta_k(.)$ as the standard basis where
\begin{eqnarray*}
\delta_k(n) =\begin{cases} 1\quad \text{if } n\equiv k (\text{mod } r)\\
 0 \quad \text{otherwise}\end{cases}
\end{eqnarray*}
for $1\leq k \leq r$. Another basis for $\Dr$ is given by the functions $e_k$ defined as $e_k(n)=\exp(\frac{2\pi ink}{r})$ for $1\leq k \leq r$. So every $f\in\Dr$ has a Fourier expansion of the form
\begin{equation}\label{fourier1}
 f(n)=\sum\limits_{1\leq k \leq r}g(k)e_k(n)
\end{equation}
where the uniquely determined Fourier coefficients  are given by the expression
\begin{equation}\label{fourier2}
 g(n)=\frac{1}{r}\sum\limits_{1\leq k \leq r}f(k)e_k(-n).
\end{equation}
Since $f, e_k\in\Dr$, so is $g$.
For an $f\in\Dr$, its DFT is given by the function
\begin{equation}\label{fourier3}
 \hat{f}(n) = \sum\limits_{1\leq k \leq r}f(k)e_k(-n).
\end{equation}

Comparing the equations (\ref{fourier2}) and (\ref{fourier3}), we get $\hat{f}=rg$. Since $f, e_k\in\Dr$,  $\hat{f}\in\Dr$. The Inverse Discrete Fourier Transform (IDFT) gives
\begin{equation}
 f(n)= \frac{1}{r}\sum\limits_{1\leq k \leq r}\hat{f}(k)e_k(n).
\end{equation}
Further, we have $\hat{\hat{f}}=rf$.
Details of the above statements can be found in \cite[Chapter~8]{tom1976introduction}.

For $f\in\Dr$, we have the following Parseval identity:
\begin{equation}
 \sum\limits_{n=1}^r|\hat{f}(n)| ^2= r \sum\limits_{n=1}^r|f(n)|^2 \label{parseval}.
\end{equation}
A proof for this identity can be found in  \cite[Chapter~4]{montgomery2006multiplicative}.

The usual Dirichlet convolution is denoted using the operator $*$. Therefore, for $f,g\in\mathcal{F}, n\in\N$, we have $$(f*g)(n)=\sum\limits_{d|n}f(d)g(n/d).$$ Note that the set of all $f\in\mathcal{F}$ with $f(1)\neq 0$ forms a group with respect to this convolution. The identity for this convolution is the function $\varepsilon$ defined by
\begin{equation*}
 \varepsilon(n) = [1/n]
\end{equation*}
where for a real number $x$, $[x]$ denotes the greatest integer not exceeding $x$. The M\"{o}bius function $\mu$ and the constant function $\mathbf{1}$ satisfies $\mu * \mathbf{1} = \varepsilon$.

For $f,g\in\Dr$, the Cauchy convolution operator $\otimes$ is defined as $$(f\otimes g)(n) = \sum\limits_{k = 1}^{r}f(k)g(r-k).$$ The DFT of the Cauchy convolution of $f$ and $g$ satisfies $\widehat{f\otimes g}=\hat{f}\hat{g}$ and $\hat{f}\otimes \hat{g} = r\widehat{fg}$. See \cite[Chapter 2]{terras1999fourier} for a proof of these properties.

\section{DFT of $(r,s)$-even functions}
 Since $(n+r^s,r^s)_s = (n,r^s)_s$, as we have already stated in the beginning, $(r,s)$-even  functions are $r^s$ periodic. That is, $\Brs\subset\Drs$. Consider the functions $g_d$ defined by 
\begin{eqnarray*}
 g_d(n) =\begin{cases}1 \text{ if } (n, r^s)_s=d^s\\
 0 \text{ otherwise }\end{cases}.
\end{eqnarray*}
It can be easily verified that these $\tau(r)$ functions form a basis of $\Brs$. By $c_{r,s}$, we denote the generalization of the Ramanujan sum defined as
\begin{equation}
 c_{r,s}(n)=\sum\limits_{\substack{j=1\\(j,r^s)_s=1}}^{r^s}\exp\left(\frac{2\pi i nj}{r^s}\right).
\end{equation}

For the generalized Ramanujan sum, the following identity was proved in \cite[Section 2]{cohen1949extension}:

\begin{equation}\label{eq:rs-gen-mu}
 c_{r,s}(n)= \sum_{d|(n,r^s)_s} d^s\mu(\frac{r}{d}).
\end{equation}

The following lemma was proved by this author in \cite{namboothiri4}. 
\begin{lemma}
 Let $e|n$. Then $c_{e,s}$ is $(n,s)$-even. That is, $ c_{e,s}(m) = c_{e,s}\left((m,n^s)_s\right)$.
 \end{lemma}

Therefore, the $\tau(r)$ functions $c_{d,s}$ for $d|r$ are in $\Brs$. Let $J_s(n)$ denote the Jordan totient function defined to be the number of  positive integers $k$ less than or equal to $n$ with $(k,n)_s=1$. McCarthy proved that
\begin{res}{\cite[Theorem 1]{mccarthy1960generation}}\label{mccarthyalpha}
An arithmetic function $f\in\Brs$ has a unique representation of the form 
\begin{equation}
 f(n)=\sum\limits_{d|r}\alpha_f(d)c_{d,s}(n)
\end{equation}
where 
\begin{eqnarray}
 \alpha_f(d) &=& \frac{1}{r^s}\sum\limits_{e|r}f(e^s)c_{\frac{r}{e},s}\left(\left(\frac{r}{d}\right)^s\right)\label{alpha_crs}\\
 &=& \frac{1}{r^sJ_s(d)}\sum\limits_{m=1}^{r^s}f(m)c_{d,s}(m).
\end{eqnarray}
\end{res}

So we see that $c_{d,s}$ form another basis for $\Brs$. Write  $e(n)=\exp(2\pi i n)$.
If $f\in\Brs$, its DFT is
\begin{eqnarray}
 \hat{f}(n) &=& \sum_{k(\text{mod } r^s)} f(k )e(-\frac{kn}{r^s})\nonumber\\
 &=& \sum_{d|r}f(d^s)\sum\limits_{\substack{1\leq j\leq \frac{r^s}{d^s} \\ (j, \frac{r^s}{d^s})_s = 1}}e(-\frac{jd^sn}{r^s})\nonumber\\
 &=& \sum_{d|r}f(d^s)c_{\frac{r}{d},s}(n)\label{fhat_crs}.
\end{eqnarray}
Since $f$ and $c_{\frac{r}{d},s}(.)$ are $(r,s)$-even , $\hat{f}(n)$ is also $(r,s)$-even . 
On comparing equations(\ref{alpha_crs}) and (\ref{fhat_crs}), we get
\begin{equation}
 r^s\alpha_f(d) = \hat{f}\left(\left(\frac{r}{d}\right)^s\right).
\end{equation}

Note also that for $f\in\Brs$ we have $\hat{\hat{f}}=r^sf$ and so 
\begin{equation}
f(n)= \frac{1}{r^s}\sum_{d|r}\hat{f}(d^s)c_{\frac{r}{d},s}(n) \label{f_crs}.
\end{equation}

\section{DFT of the generalized Ramanujan sum}
Consider the function $\rho_{r,s}$ defined by 
\begin{equation}
 \rho_{r,s}(n) =\begin{cases} 1 \text{ if }(n,r^s)_s=1\\
0   \text{ otherwise }\end{cases}.
\end{equation}
Note that $ ((n,r^s)_s,r^s)_s = (n,r^s)_s $. So if  $\rho_{r,s}(n)=1$, then $(n,r^s)_s = 1 \Rightarrow  ((n,r^s)_s,r^s)_s = 1 \Rightarrow \rho_{r,s}((n,r^s)_s) = 1$. Similary if  $\rho_{r,s}(n)=0$, then $(n,r^s)_s \neq  1 \Rightarrow  ((n,r^s)_s,r^s)_s \neq  1 \Rightarrow \rho_{r,s}((n,r^s)_s) = 0$.
So it follows that $\rho_{r,s}$ is $(r,s)$-even. 

Now the DFT of $\rho_{r,s}$ (as an $(r,s)$-even  function) is 
\begin{eqnarray}
\hat{\rho}_{r,s}(n) &=& \sum_{d|r}\rho_{r,s} (d^s)c_{\frac{r}{d},s}(n)\nonumber\\
 &=& c_{r,s}(n)\label{rhoDFT}
\end{eqnarray}
because for $d|r$, $(d^s, r^s)_s=1$ if and only if $d=1$. For any other $d|r$, $\rho_{r,s} (d^s)$ is $0$.

Now $$\hat{\hat{\rho}}_{r,s}(n)=r^s\rho_{r,s}(n)$$ and so $\hat{c}_{r,s}(n) = r^s\rho_{r,s}(n)$.

\section{Some applications of DFT}
In this section, we give new proofs for certain known results using DFT. Our proofs are shorter than the already known proofs.

Cohen in \cite{cohen1956extensionof} gave the following result mainly using the orthogonality relations of the generalized Ramanujan sum. We give an alternate proof  using DFT.
\begin{prop}\cite[Theorem 2]{cohen1956extensionof}
\newline
 $\sum\limits_{d|r}c_{d,s}(n) c_{r,s}(\left(\frac{r}{d}\right)^s)=\begin{cases}
                                                                  r^s \text{ if }(n,r^s)_s=1\\
                                                                  0 \text{ otherwise}
                                                                 \end{cases}
$
\begin{proof}
 We have already observed that $r^s f(n)= \sum\limits_{d|r}\hat{f}(d^s)c_{\frac{r}{d},s}(n) $ and that $\hat{\rho}_{r,s}(n) = c_{r,s}(n)$. So 
 \begin{eqnarray*}
  r^s \rho_{r,s}(n) &=& \sum\limits_{d|r}c_{r,s}(d^s)c_{\frac{r}{d},s}(n)\\
   &=& \sum\limits_{d|r}c_{r,s}\left(\frac{r^s}{d^s}\right)c_{d,s}(n).
 \end{eqnarray*}
  By definition, $\rho_{r,s}(n)=1$ if $(n,r^s)_s=1$ and 0 otherwise. Now the claim follows.
\end{proof}

\end{prop}

Our next result generalizes Theorem 4 of \cite{nicol1954sterneck}:
\begin{prop}
 For any even number $r$ and positive integer $s$, 
 \begin{equation*}
  \sum\limits_{d|r}(-1)^{d^s}c_{\frac{r}{d},s}(n)=\begin{cases}
                                                   r^s \text{ if } n\equiv \frac{r^s}{2}(\text{mod }r^s)\\
                                                   0 \text{ otherwise}
                                                  \end{cases}.
 \end{equation*}
\begin{proof}
 Consider the function $f(n)=(-1)^{n^s}$.
 
 \noindent\textbf{Claim:} $f$ is $(r,s)-$even.
 \begin{enumerate}[{Case }1:]
  \item For $n$ odd, $n^s$ is odd and so $(-1)^{n^s} = -1$. Also $(n^s,r^s)_s$ is odd, so that $(-1)^{(n^s,r^s)_s}=-1$.
  \item For $n$ even, $n^s$ is even and so $(-1)^{n^s} = 1$. Also $(n^s,r^s)_s$ is even (as atleast $2^s|n^s$ and $2^s|r^s$),  so $(-1)^{(n^s,r^s)_s}=1$.
 \end{enumerate}
 Hence in either case, $f(n)=(-1)^{n^s}=(-1)^{(n^s,r^s)_s}$. So the DFT of $(r,s)$-even  functions can be applied to $f$. Hence
 \begin{eqnarray*}
  \hat{f}(n) &=& \sum\limits_{d|r}f(d^s)c_{\frac{r}{d},s}(n)\\
  &=&\sum\limits_{d|r}\left((-1)^{d^s}\right)^{d^s}c_{\frac{r}{d},s}(n)\\
  &=&\sum\limits_{d|r}(-1)^{d^s} c_{\frac{r}{d},s}(n).
 \end{eqnarray*}
 But by the DFT of $r^s$ periodic functions, we have
 \begin{eqnarray*}
  \hat{f}(n) &=& \sum\limits_{k = 1}^{r^s}f(k) \exp\left(\frac{-2\pi i kn}{r^s}\right)\\
   &=& \sum\limits_{k = 1}^{r^s} (-1)^{k^s} \exp\left(\frac{-2\pi i kn}{r^s}\right)\\
   &=& \sum\limits_{k = 1}^{r^s} \left(-\exp\left(\frac{-2\pi i n}{r^s}\right)\right)^k.
 \end{eqnarray*}
If $\frac{n}{r^s} = \frac{1}{2}$, then  we have to sum $-\exp(-\pi i) = 1$ from $k = 1$ to $r^s$  which then equals $r^s$. If $\frac{n}{r^s}$ is an integer then since $r^s$ is even, $k^{\text{th}}$ powers of $-\exp\left(\frac{-2\pi i n}{r^s}\right) = -1$ alternates between $+1$ and $-1$  an equal number of times  cancelling each other. In any other case, $-\exp\left(\frac{-2\pi i n}{r^s}\right)$ is a non integer root of unity and the sum the powers will give 0 since $r^s$ is even.

\end{proof}
\end{prop}

Recall the notation $\alpha_f$ in the McCarthy's theorem (\ref{mccarthyalpha}). We prove
\begin{thm}\label{CauchyProdInBrs}
 Let $f,g\in\Brs$. Then
 \begin{enumerate}
  \item $f\otimes g\in\Brs$.
  \item $\alpha_{f\otimes g}(d) = r^s\alpha_f(d)\alpha_g(d)$.
 \end{enumerate}
\begin{proof}
By definition $$ (f\otimes g)(n) = \sum\limits_{k(\text{mod }r^s)} f(k) g(n-k).$$
   Given that $f,g\in\Brs\subset \Drs$. So $\widehat{f\otimes g}=\hat{f}\hat{g}$ and
 \begin{eqnarray*}
  \widehat{f\otimes g}((n,r^s)_s) &=& \hat{f}((n,r^s)_s) \hat{g}((n,r^s)_s)\\
  &=& \hat{f}(n) \hat{g}(n)\\
  &=& \widehat{f\otimes g}(n)
  \end{eqnarray*}
 which implies that $\widehat{f\otimes g}$ is $(r,s)$-even . Hence  $\widehat{\widehat{f\otimes g}} = r^s (f\otimes g)$ is $(r,s)$-even  and $f\otimes g$ is $(r,s)$-even .
 
Now 
\begin{eqnarray*}
 \alpha_{f\otimes g}(d) &=& \frac{1}{r^s}(\widehat{f\otimes g})(\left(\frac{r}{d}\right)^s)\\
 &=& \frac{1}{r^s}\hat{f}(\left(\frac{r}{d}\right)^s)\hat{g}(\left(\frac{r}{d}\right)^s)\\
 &=& \alpha_f(d) r^s \alpha_g(d).
\end{eqnarray*}

\end{proof}

\end{thm}

Consider the linear congruence equation.
\begin{equation}
 a_1 x_1+\ldots +a_k x_k\equiv n \,(\text{mod } r) \label{gen_lin_cong}.
\end{equation}
If we seek  solutions for this equation with some restrictions on the solution set, like $(x_i,r)=t_i$ for  $1\leq i \leq k$ where $t_i$ are given positive divisors of $r$, then it is called to be a restricted linear congruence. Many authors have attempted to solve such restricted congruences with varying conditions. Cohen \cite{cohen1956extensionof} gave a formula for the number of solutions $N_{r,s}$ of this congruence equation with $a_i=1, t_i=1$, $r$ replaced with $r^s$ and the restriction $(x_i, r^s)_s=1$. Here we provide an alternate method for arriving at his formula using the techniques in DFT. We would like to further remark that the author himself has proved two generalizations of the next theorem in \cite{namboothiri4} and  \cite{namboothiri2018restricted}.
\begin{thm}\cite[Theorem 12, Theorem 12']{cohen1956extensionof}
 Consider the linear congruence equation $x_1+\ldots+x_k\equiv n \,(\text{mod } r^s)$. With the restrictions $(x_i,r^s)_s=1$, the number of solutions of this congruence equation is
 \begin{equation}
  N_{r,s}(n,k) = \frac{1}{r^s}\sum\limits_{d|r} \left(c_{r,s}(\left(\frac{r}{d}\right)^s)\right)^k c_{d,s}(n).
 \end{equation}
\begin{proof}
 The function $\rho_{r,s}\in\Brs\subset\Drs$ plays the lead role here. The Cauchy convolution of $\rho_{r,s}$ taken $k$ copies is
 $$(\rho_{r,s}\otimes\ldots\otimes\rho_{r,s})(n)=\sum\limits_{a_1+\ldots+a_k=n}\rho_{r,s}(a_1)\ldots\rho_{r,s}(a_k)$$ 
 By definition, 
 \begin{equation}
 \rho_{r,s}(a) =\begin{cases} 1 \text{ if }(a,r^s)_s=1\\
  0 \text{ otherwise }\end{cases}.
\end{equation}
Therefore, $ N_{r,s}(n,k) = (\rho_{r,s}\otimes\ldots\otimes\rho_{r,s})(n)$. Since $\rho_{r,s}\in\Brs$, by theorem (\ref{CauchyProdInBrs}) $ N_{r,s}\in\Brs$ and so 
$\widehat{N_{r,s}} = \hat{\rho}_{r,s}\ldots \hat{\rho}_{r,s}$ -- $k$ times which is equal to $({\hat{\rho}_{r,s}})^k$. By equation(\ref{rhoDFT}), we have $\hat{\rho}_{r,s} =c_{r,s}$. Therefore $\widehat{N_{r,s}} = \left(c_{r,s}\right)^k$. By (\ref{f_crs}), we have
\begin{eqnarray*}
 N_{r,s}(n,k) &=& \frac{1}{r^s}\sum\limits_{d|r} (c_{r,s}(d^s))^k c_{\frac{r}{d},s}(n)\\
 &=& \frac{1}{r^s}\sum\limits_{d|r} \left(c_{r,s}(\left(\frac{r}{d}\right)^s)\right)^k c_{d,s}(n).
\end{eqnarray*}

\end{proof}
\end{thm}
Cohen gave an inversion formula  for $r$-even functions in \cite[Theorem 3]{cohen1958representations}. We give an analogous formula for $(r,s)$-even  functions in the next theorem.
\begin{thm}[Inversion formula for $(r,s)$-even  functions]
 Let $f,g$ be $(r,s)$-even  functions with $f$ defined as 
 \begin{equation}
  f(n)=\sum\limits_{d|r}g(d)c_{d,s}(n).
 \end{equation}
Then 
\begin{equation}
 g(m)=\frac{1}{r^s}\sum\limits_{d|r}f(\frac{r^s}{d^s}) c_{d,s}(n)\text{ where } m^s = \frac{r^s}{(n,r^s)_s}.
\end{equation}
\begin{proof}
 Let $G(n)=g( m)$. then $G$ is $(r,s)$-even  and so
 \begin{eqnarray*}
  \hat{G}(n) &=& \sum\limits_{d|r}G(d^s)c_{\frac{r}{d},s}(n)\\
  &=& \sum\limits_{d|r}G(r^s/d^s)c_{d,s}(n)\\
  &=& \sum\limits_{d|r}g(d)c_{d,s}(n)\\
  &=&f(n).
 \end{eqnarray*}
So  $r^s G(n) = \hat{\hat{G}}(n) =\hat{f}(n) = \sum\limits_{d|r} f(r^s/d^s)c_{d,s}(n)$ and the theorem follows.
 
\end{proof}

\end{thm}

Following result is a consequence of the Parseval formula (\ref{parseval}) and DFT. 
\begin{prop}
 If $f\in\Brs$, then $$\sum\limits_{n=1}^{r^s}|\hat{f}(n)|^2 = r^s \sum\limits_{d|r} |f(d^s)|^2 J_s(r^s/d^s).$$
 \begin{proof}
  By Parseval formula, 
  \begin{eqnarray*}
 \sum\limits_{n=1}^{r^s}|\hat{f}(n)| ^2 &=& r^s \sum\limits_{n=1}^{r^s}|f(n)|^2 \\
 &=& r^s \sum\limits_{n=1}^{r^s}|f((n,r^s)_s)|^2 \text{ since $f$ is $(r,s)$-even }\\
 &=& r^s \sum\limits_{(n,r^s)_s = d^s}|f(d^s)|^2 \times \text{number of $n$ with $1\leq n\leq r^s$ and }(n,r^s)_s = d^s \\
  &=& r^s \sum\limits_{d|r}|f(d^s)|^2 J_s(r^s/d^s).
\end{eqnarray*}
 \end{proof}

\end{prop}

\begin{cor}
 $$\sum\limits_{n=1}^{r^s} (c_{r,s}(n))^2 = r^sJ_s(r^s).$$
 \begin{proof}
  In the above propositiion, let $f=\rho_{r,s}$. Then $\hat{f}=c_{r,s}$. Since $c_{r,s}$ is an integer, by the above proposition 
  $$\sum\limits_{n=1}^{r^s} (c_{r,s}(n))^2 =r^s \sum\limits_{d|r}\rho_{r,s}(d^s)^2 J_s(r^s/d^s).$$
  In this summation, $\rho_{r,s}(d^s)$ is nonzero only if $(d^s,r^s)_s=1$ and in that case $d=1$. But $\rho_{r,s}(1)=1$. So the right side becomes $r^sJ_s(r^s)$.
 \end{proof}

 \end{cor}
\section{DFT of multiplicative $(r,s)$-even  functions}
Recall that an arithmetic function $f$ of a single variable is multiplicative if $f(mn)=f(m)f(n)$ for every $m,n\in\N$ with $(m,n)=1$. It is completely multiplicative if $f(mn)=f(m)f(n)$ for every $m,n\in\N$. A function $f$ of two variables is multiplicative if $f(n_1 n_2,r_1 r_2)=f((n_1,r_1)) f((n_2,r_2))$ whenever $(n_1,n_2)=(n_1,r_2) = (n_2,r_1)=(r_1,r_2) = 1$. (This GCD condition is equivalent to writing $(n_1 r_1, n_2 r_2)=1$.)

 By Theorem (1) in \cite{cohen1949extension},  $c_{r,s}$ is multiplicative in $r$. It follows from the next theorem that $c_{r,s}(n)$ is multiplicative as a function of two variables $r$ and $n$.
 \begin{thm}\label{multsequence}
  Let $(f_r)_{r\in\N}$ be a sequence of functions. Assume that $f_r\in\Brs$ and that $r\rightarrow f_r(n)$ is multiplicative in $r$. Then $f_r(n)$ viewed as a function of two variables $f((n,r)):\N^2\rightarrow \C$ is multiplicative.

  \begin{proof}
   Assume that $(m,n) = (m,r) = (n,q) = (q,r) = 1$ for some positive integers $m,n,q,r$. Then
   \begin{eqnarray*}
    f_{qr}(mn) &=& f_q(mn) f_r(mn)\\
    &=& f_q((mn,q^s)_s) f_r((mn,r^s)_s)\\
    &=& f_q((m,q^s)_s) f_r((n,r^s)_s)\\
    &=& f_q(m) f_r(n).
   \end{eqnarray*}
  \end{proof} 
    \end{thm}
Every multiplicative function of two variables $n$ and $r$ satisfies the quasi multiplicative property $f((n,r)) f((n',r)) = f((1,r)) f(nn',r))$ whenever $(n,n') = 1$. This statement has been proved by many authors. See \cite[Theorem 65]{sivaramakrishnan1988classical} for example.
Since every $f\in\Brs$ is multiplicative in two variables, the following result is immediate.
\begin{cor}
 Let $f_r\in\Brs$. Then $f_r(m) f_r(n) = f_r(1) f_r(mn)$ for $(m,n)=1$. In particular, $f_r(n)$ is multiplicative as a function of the variable $n$ if and only if $f_r(1)=1$.
\end{cor}

Next we prove that DFT of an $(r,s)$-even  multiplicative function is multiplicative.
\begin{thm}\label{multsequencedft}
 Let $f_r\in\Brs$ for $r\in\N$. 
 \begin{enumerate}
  \item If $r\rightarrow f_r(n)$ is multiplicative, then $r\rightarrow \hat{f}_r(n)$ is multiplicative.
  \item $\hat{f}:\N^2\rightarrow \C$ is multiplicative as a function of 2 variables.
 \end{enumerate}
 
 \begin{proof}
 \begin{enumerate}
  \item 
  Let $(q,r)=1$. By definition of DFT, we have 
  \begin{eqnarray*}
   \hat{f}_{qr}(n) &=& \sum_{d|qr}f_{qr}(d^s)c_{\frac{qr}{d},s}(n)\\
   &=& \sum_{a|q,b|r}f_{qr}(a^s b^s)c_{\frac{q}{a}\frac{r}{b},s}(n)\text{ where } d=ab, (a,b)=1\\
   &=& \sum_{a|q,b|r}f_{q}(a^s b^s) f_{r}(a^s b^s) c_{\frac{q}{a}\frac{r}{b},s}(n)\\
   &=& \sum_{a|q,b|r}f_{q}((a^s b^s,q^s)_s) f_{r}((a^s b^s,r^s)_s) c_{\frac{q}{a}\frac{r}{b},s}(n)\\
   &=& \sum_{a|q,b|r}f_{q}((a^s,q^s)_s) f_{r}((b^s,r^s)_s) c_{\frac{q}{a}\frac{r}{b},s}(n)\\
   &=& \sum_{a|q,b|r}f_{q}(a^s) f_{r}(b^s) c_{\frac{q}{a}\frac{r}{b},s}(n)\\
   &=& \sum_{a|q,b|r}f_{q}(a^s) f_{r}(b^s) c_{\frac{q}{a},s}(n) c_{\frac{r}{b},s}(n)\\
   &=& \sum_{a|q}f_{q}(a^s)  c_{\frac{q}{a},s}(n) \times \sum_{b|r} f_{r}(b^s)  c_{\frac{r}{b},s}(n)\\
   &=& \hat{f}_q(n) \hat{f}_r(n).
     \end{eqnarray*}
\item Note that $r\rightarrow\hat{f}_r(n)$ is multiplicative in $r$ and $\hat{f}_r\in\Brs$. Then the statement follows from theorem (\ref{multsequence}).
\end{enumerate}
 \end{proof}

\end{thm}

The following definition generalizes the completely $r$-even functions given by Cohen in \cite{cohen1958representations}.
 \begin{defn}
A sequence of arithmetic functions $(f_r)_{r\in\N}$ is said to be completely $(r,s)$-even if there exists an $F\in\mathcal{F}$ such that $f_r(n)=F((n,r^s)_s)$ for all $r\in\N$.
 \end{defn}
It can be easily seen that $f_r$ is therefore $(r,s)$-even . Note that $(n,r^s)_s$ is multiplicative in $n$ as well as $r$ independently.  So if $F$ is multiplicative as a single variable function, then $f_r(n)=f(n,r)$ is multiplicative in $n$ and $r$ independently. We list some more properties of the completely $(r,s)$-even functions in the next result.
\begin{thm}
 Let $(f_r)_{r\in\N}$ be a sequence of functions. Assume that 
 $(f_r)_{r\in\N}$ is completely $(r,s)$-even with $f_r(n)=F((n,r^s)_s)$ for some multiplicative function $F\in\mathcal{F}$. Then
\begin{enumerate}
 \item $f_r(n)$ is multiplicative in $r$.
 \item $f_r(n)$ is multiplicative in $n$.
 \item $f_r\in\Brs$.
 \item $f_r(n)=f(n,r)$ is multiplicative as a function of two variables.
 \item $r\rightarrow \hat{f}_r(n)$ is multiplicative.
 \item $n\rightarrow \hat{f}_r(n)$ is multiplicative if and only if $\hat{f}_r(1)=1$.
\end{enumerate}
\begin{proof}
 We have already seen that the statements (1),(2), and (3) are true. Now (4) follows from Theorem (\ref{multsequence}) and (5) follows from Theorem(\ref{multsequencedft}). To see (6), note that  $\hat{f}(n) = \sum\limits_{d|r} f(d^s)c_{\frac{r}{d},s}(n)$ and $\hat{f}_r(m) \hat{f}_r(n)  = \hat{f}_r(1)  \hat{f}_r(mn) $ for any $m,n$ with $(m,n)=1$ and so $\hat{f}_r$ is multiplicative if and only if $\hat{f}_r(1) =1$. 

\end{proof}
\end{thm}

\section{A generalization of H\"{o}lder's identity}
For the classical Ramanujan sum $c_r(n)$, O. H\"{o}lder proved \cite{holder1936theorie} the identity $$c_r(n)=\frac{\phi(r)\mu(m)}{\phi(m)}$$ where $m=\frac{r}{(n,r)}$. Generalizing this, Cohen proved in \cite{cohen1956extensionof} that the generalized Ramanujan sum  $c_{r,s}(n)$ satisfies
\begin{equation}\label{ghi}
 c_{r,s}(n) = \frac{J_s(r)\mu(m)}{J_s(m)}\text{ where } m^s = \frac{r^s}{(n,r^s)_s}.
\end{equation}
We here give a generalization of this identity to a particular class of functions. Recall that a multiplicative function $f$ is strongly multiplicative if $f(p^a)=f(p)$ for any prime $p$ and  $a > 0$. We prove 
\begin{thm}\label{ggholder}
 Let $(f_r)_{r\in\N}$ be a sequence of complete $s$-even functions with $f_r(n)=F((n,r^s)_s)$ for some strongly multiplicative function $F$. Suppose  that $F(p)\neq 1-p^s$ for any prime $p$. Then 
 \begin{equation}\label{ggholderidentity}
  \hat{f}_r(n) = \frac{(F*\mu)(m) (F* J_s) (r)}{(F* J_s)(m)}\text{ where } m^s=\frac{r^s}{(n,r^s)_s}.
 \end{equation}

\end{thm}

In the proof of the theorem, we need two results established by  Cohen.
\begin{res}\cite[Theorem 3]{cohen1949extension}\label{cpacases}
 If $p$ is a prime and $\lambda>0$, then 
 \begin{equation}
  c_{p^{\lambda},s}(n)=\begin{cases}
                p^{s\lambda} - p^{s(\lambda-1)} &\text{ if }p^{s\lambda}|n\\
                - p^{s(\lambda-1)} &\text{ if }p^{s\lambda}\nmid n, p^{s(\lambda-1)} | n\\
                0 &\text{ otherwise}
               \end{cases}.
 \end{equation}

\end{res}

\begin{res}\cite[Corollary 2.2]{cohen1950extension}\label{cpsumcases}
 \begin{equation}
  \sum\limits_{d|r}c_{d,s}(n) = \begin{cases}
                                r^s &\text{ if } n\equiv 0\,(\text {mod }r^s)\\
                                0 &\text{ otherwise } 
                               \end{cases}.
 \end{equation}

\end{res}

We will also use the fact that the Jordan totient function $J_s$ defined by $$J_s(n)=n^s\prod\limits_{p|n}(1-p^{-s})$$ satisfies (\cite[Corollary 7.1]{cohen1956extensionof}) \begin{equation}
                                \sum\limits_{d|n}J_s(d) = n^s.
                               \end{equation}

We now proceed to prove our theorem.
\begin{proof}[Proof of theorem (\ref{ggholder})]
Note that $r\rightarrow\hat{f}_r$ is multiplicative. Since  $F,\mu$ and  $J_s$ are all multiplicative, the expression on the right side of the  identity in theorem (\ref{ggholder}) is multiplicative. Therefore it is enough to prove the identity for prime powers $p^a$ for arbitrary prime $p$ and $a > 0$. Let us check the left hand side of the identity (\ref{ggholderidentity}):

\begin{eqnarray*}
 \hat{f}_{r}(n) &=& \sum\limits_{d|r}f_{r}(d^s)c_{\frac{r}{d},s}(n)\\
 &=& \sum\limits_{d|r}F((d^s,r^{s})_s)c_{\frac{r}{d},s}(n)\\
 &=& \sum\limits_{d|r}F(d^s)c_{\frac{r}{d},s}(n).
\end{eqnarray*}
Put $r=p^a$. Then

\begin{eqnarray*}
 \hat{f}_{p^a}(n) &=& F(1)c_{p^a,s}(n)+\sum\limits_{\substack{d|p^{a}\\d\neq 1}}F(d^s)c_{\frac{p^{a}}{d},s}(n)\\
  &=& F(1)c_{p^a,s}(n)+F(p)\sum\limits_{d|p^{a-1}}c_{\frac{p^{a-1}}{d},s}(n).
  \end{eqnarray*}
$F$ being multiplicative and identically not zero, $F(1)=1$. Now the the first term on the right side sum in the above expression has to be determined using Result (\ref{cpacases}) and the second term using the result (\ref{cpsumcases}).  We  see that

\begin{eqnarray}
 \hat{f}_{p^a}(n) &=& \begin{cases}
 p^{sa} - p^{s(a-1)} + F(p)p^{s(a-1)}  &\text{ if }p^{sa}|n\\                 
   - p^{s(a-1)} + F(p)p^{s(a-1)} &\text{ if }p^{sa}\nmid n, p^{s(a-1)} | n\\
      0 + F(p) 0 &\text{ if } p^{s(a-1)} \nmid n
                   \end{cases}\nonumber\\
&=& \begin{cases}
 p^{s(a-1)}(F(p) + p^s-1) & \text{ if }p^{sa}|n\\                 
    p^{s(a-1)} (F(p)-1) &\text{ if }p^{sa}\nmid n, p^{s(a-1)} | n\\
      0 &\text{ if } p^{s(a-1)} \nmid n
                   \end{cases}.
  \end{eqnarray}
  Let us now evaluate the right hand side of the identity. We will evaluate $(F* J_s) (p^a)$ first.

\begin{eqnarray}
 (F*J_s)(p^a) &=&F(1)J_s(p^a)+\sum\limits_{\substack{d|p^{a}\\d\neq 1}}F(d)J_s(\frac{p^{a}}{d})\nonumber\\
 &=& p^{sa}(1-p^{-s}) + F(p)\sum\limits_{d|p^{a-1}}J_s(\frac{p^{a-1}}{d})\nonumber\\
 &=&  p^{sa}-p^{s(a-1)} + F(p)p^{s(a-1)}\nonumber\\
 &=&  p^{s(a-1)}[p^{s}-1 + F(p)].
 \end{eqnarray}
 We have $m^s=\frac{p^{sa}}{(n,p^{sa})_s}$. Therefore 
 \begin{eqnarray*}
    m = \begin{cases}
                1 &\text{ if }p^{sa}|n\\
                 p &\text{ if }p^{sa}\nmid n, p^{s(a-1)} | n\\
                p^{\delta} &\text{ otherwise where }0\leq \delta \leq a-2
                \end{cases}.
 \end{eqnarray*}
We will evaluate the values of $(F*\mu)(m)$ and $(F* J_s)(m)$ in three different cases.
\begin{enumerate}[{Case }1 : ]
 \item When  $p^{sa}|n$, $(F*\mu)(m) = (F*\mu)(1) = F(1)\mu(1) = 1$. Now $(F*J_s)(m) = F(1)J_s(1) = 1$.
 \item When $p^{sa}\nmid n, p^{s(a-1)} | n$, $(F*\mu)(m)=(F*\mu)(p) = F(1)\mu(p)+F(p)\mu(1) = F(p)-1$. Now $(F*J_s)(m)=F(1)J_s(p)+F(p)J_s(1) = F(p)+p^s-1$.
 \item When $p^{s(a-2)}\nmid n$, $(F*\mu)(m)=(F*\mu)(p^{\delta}) = F(p^{\delta})\mu(1)+F(p^{\delta-1})\mu(p) = F(p)\mu(1)+F(p)\mu(p)  = 0$. 
\end{enumerate}
Now if we substitute these values to the identity, given that $p^{s}-1 + F(p)\neq 0$, the identity (\ref{ggholderidentity}) follows.
\end{proof}

Take $F(n)=\varepsilon(n)=\left[\frac{1}{n}\right]$, the identity for the Dirichlet convolution. Then $F$ is strongly multiplicative. Now $F((n,r^s)_s)=\begin{cases}1\text{ if }(n,r^s)_s)=1\\
0\text { otherwise}                                                                                                                                                                                                     \end{cases} = \rho_{r,s}(n)$. Put $f_r(n)= \rho_{r,s}(n)$. Then $f_r$ satisfies the conditions of the above theorem. The DFT of $f_r$ is nothing but $c_{r,s}(n)$. So we see that $$c_{r,s}(n)=\hat{f}_r(n) = \frac{\mu(m)J_s(r)}{J_s(m)}$$ which is the generalized H\"{o}lder identity given in equation (\ref{ghi}).

\end{document}